\documentclass[11pt]{amsart}
\usepackage{amsmath}
\usepackage{amssymb, verbatim}
\usepackage{pstricks,pst-node,pst-plot}

\title{K\"ahler currents and null loci}
\author[T.C. Collins]{Tristan C. Collins}
\address{Department of Mathematics, Harvard University, 1 Oxford Street, Cambridge, MA 02138}
  \email{tcollins@math.harvard.edu}
\author[V. Tosatti]{Valentino Tosatti$^{*}$}
\thanks{$^{*}$Supported in part by a Sloan Research Fellowship and NSF grants DMS-1236969 and DMS-1308988.}
 \address{Department of Mathematics, Northwestern University, 2033 Sheridan Road, Evanston, IL 60201}
  \email{tosatti@math.northwestern.edu}
  \dedicatory{Dedicated to D.H. Phong with admiration on the occasion of his 60th birthday.}

\theoremstyle{plain}
\newtheorem{thm}{Theorem}[section]
\newtheorem{prop}[thm]{Proposition}
\newtheorem{defn}[thm]{Definition}
\newtheorem{lem}[thm]{Lemma}
\newtheorem{cor}[thm]{Corollary}

\newtheorem{que}[thm]{Question}

\theoremstyle{definition}

\numberwithin{equation}{section}

\newcommand{\amp}{\textnormal{Amp}}
\newcommand{\Null}{\textnormal{Null}}
\newcommand{\del}{\partial}
\newcommand{\de}{\partial}
\newcommand{\dbar}{\overline{\del}}
\newcommand{\ddb}{i\del\dbar}
\newcommand{\ddt}{\frac{\del}{\del t}}

\newcommand{\ve}{\varepsilon}
\newcommand{\vp}{\varphi}
\newcommand{\ti}[1]{\tilde{#1}}
\newcommand{\tr}[2]{\mathrm{tr}_{#1}{#2}}
\newcommand{\Ric}{\mathrm{Ric}}

\renewcommand{\leq}{\leqslant}
\renewcommand{\geq}{\geqslant}
\renewcommand{\le}{\leqslant}

\renewcommand{\epsilon}{\varepsilon}
\renewcommand{\phi}{\varphi}
\newcommand{\ov}[1]{\overline{#1}}

\begin{document}
\begin{abstract}
We prove that the non-K\"ahler locus of a nef and big class on a compact complex manifold bimeromorphic to a K\"ahler manifold equals its null locus. In particular this gives
an analytic proof of a theorem of Nakamaye and Ein-Lazarsfeld-Musta\c{t}\u{a}-Nakamaye-Popa. As an application, we show that finite time non-collapsing singularities of the K\"ahler-Ricci flow on compact K\"ahler manifolds always form along analytic subvarieties, thus answering a question of Feldman-Ilmanen-Knopf and Campana. We also extend the second author's results about noncollapsing degenerations of Ricci-flat K\"ahler metrics on Calabi-Yau manifolds to the nonalgebraic case.
\end{abstract}
\maketitle
\section{Introduction}
The general topic of this paper is the study of the boundary of the K\"ahler cone of a compact K\"ahler manifold, and our goal is to understand how much cohomology classes on the boundary of the K\"ahler cone deviate from being K\"ahler.
Let $(X,\omega)$ be an $n$-dimensional compact K\"ahler manifold. If $\alpha$ is a closed real $(1,1)$ form on $X$, then we denote by $[\alpha]$ its cohomology class in $H^{1,1}(X,\mathbb{R})$.
The set of cohomology classes of K\"ahler metrics is the K\"ahler cone $\mathcal{K}\subset H^{1,1}(X,\mathbb{R})$, which is an open convex cone. Classes in its closure $\ov{\mathcal{K}}$ are called nef (or numerically effective). A class $[\alpha]$ is said to have positive self-intersection if $\int_X\alpha^n>0$. Cohomology classes which are nef and with positive self-intersection are quite close to being K\"ahler. In fact, thanks to a fundamental result of Demailly-P\u{a}un \cite{DP}, such classes contain a K\"ahler current, which roughly speaking is a singular K\"ahler metric (see Section \ref{sectprel} for definitions). As a result, we will also refer to them as nef and big classes. Furthermore, thanks to Demailly's regularization theorem \cite{Dem92}, the K\"ahler current can be chosen to have analytic singularities, which in particular implies that it is a smooth K\"ahler metric on a Zariski open subset of $X$. The largest possible subset of $X$ that we can obtain in this way is called the ample locus of $[\alpha]$, denoted by $\amp(\alpha)$. Its complement $E_{nK}(\alpha)=X\backslash \amp(\alpha)$ is called the non-K\"ahler locus of $[\alpha]$, and it is a proper analytic subvariety of $X$. In a sense $E_{nK}(\alpha)$ measures the deviation from $[\alpha]$ to being K\"ahler, since
$E_{nK}(\alpha)=\emptyset$ if and only if $[\alpha]$ is K\"ahler. Ample and non-K\"ahler loci were introduced by Boucksom \cite{Bou}, and they have been much studied recently in connection with the regularity for complex Monge-Amp\`ere equations in big cohomology classes \cite{BD, BEGZ, D6}.

Given a nef and big class $[\alpha]$, there is another subset of $X$ which measures its non-K\"ahlerity, namely the null locus
$$\Null(\alpha)=\bigcup_{\int_V \alpha^{\dim V}=0}V,$$
where the union is over all positive dimensional irreducible analytic subvarieties $V\subset X$ where $[\alpha]$ integrates to zero.
Note that since $\int_X\alpha^n>0$, all such subvarieties are proper subsets of $X$. The null locus was explicitly first considered by Nakamaye \cite{Nak}, although it was probably studied earlier. Our main theorem is the following:

\begin{thm}\label{main}
Let $X$ be a compact complex manifold which is bimeromorphic to a K\"ahler manifold, let $\alpha$ be a closed smooth real $(1,1)$-form on $X$ whose class $[\alpha]\in H^{1,1}(X,\mathbb{R})$ is nef with $\int_X \alpha^n>0$. Then
\begin{equation}
E_{nK}(\alpha) = \Null(\alpha).
\end{equation}
\end{thm}

Recall that compact complex manifolds bimeromorphic to K\"ahler are said to be in Fujiki's class $\mathcal{C}$ \cite{Fu, Va}. Equivalently, $X$ is in class $\mathcal{C}$ if there is a modification $\mu:\ti{X}\to X$ with $\ti{X}$ K\"ahler.

In particular, Theorem \ref{main} proves that $\Null(\alpha)$ is a proper analytic subvariety of $X$, a fact which is far from clear from its definition (see \cite{Nak} for a proof of this in
the algebraic setting). Note that this theorem is trivially true if we drop the hypothesis that $\int_X \alpha^n>0$, in which case $E_{nK}(\alpha) = \Null(\alpha)=X$.

A simple example where this theorem applies is the following: let $(Y^n,\omega)$ be a compact K\"ahler manifold and $\pi:X\to Y$ be the blowup of a point $y\in Y$ with exceptional divisor
$E=\pi^{-1}(y)\cong \mathbb{CP}^{n-1}$. Let $[\alpha]=[\pi^*\omega]$, so $[\alpha]$ is nef and has positive self-intersection, with $\Null(\alpha)=E$. It is well-known that there exists a smooth metric $h$ on the line bundle $\mathcal{O}(E)$ with curvature $R(h)$ such that $\pi^*\omega-\ve R(h)$ is a K\"ahler metric on $X$ for all $\ve>0$ small. If $s$ is a section of $\mathcal{O}(E)$ that vanishes along $E$ to order $1$, then for any $\ve$ small $\pi^*\omega+\ve\ddb\log|s|^2_h$ is a K\"ahler current on $X$ in the class $\alpha$ singular precisely along $E$.

This theorem has several diverse applications, which we now describe. The most direct corollary is an analytic proof of the following theorem in algebraic geometry,
due to Nakamaye \cite{Nak} in the case of $\mathbb{Q}$-divisors and to Ein-Lazarsfeld-Musta\c{t}\u{a}-Nakamaye-Popa \cite[Corollary 5.6]{ELMNP} in general:

\begin{cor}[\cite{ELMNP, Nak}]\label{locus}
Let $X$ be an $n$-dimensional smooth projective variety over $\mathbb{C}$ and $D$ a nef $\mathbb{R}$-divisor on $X$ with $D^n>0$. Then the augmented base locus $\mathbb{B}_{+}(D)$ of $D$ satisfies $\mathbb{B}_{+}(D) = \Null(D)$.
\end{cor}

Indeed the initial motivation for our work was to extend this result to all compact K\"ahler manifolds. For different extensions of this result to possibly singular projective varieties, see \cite{CMM} in positive characteristic, and \cite{CL} in characteristic zero.

From Theorem \ref{main}, one can also deduce a celebrated result of Demailly-P\u{a}un \cite{DP}.  Our proof is not, by any means, independent of the work of \cite{DP}, and it relies crucially on the mass concentration technique developed by Demailly-P\u{a}un.   We are, however, able to avoid the complicated induction on dimension used in \cite{DP}, which requires working on singular complex analytic spaces and using the results of P\u{a}un in \cite{Paun}.

\begin{cor}[\cite{DP}, Theorem 0.1]\label{demap}
Let $X$ be a compact K\"ahler manifold, and denote by $\mathcal{K}$ the K\"ahler cone of $X$.  Set
\begin{equation*}\begin{split}
 \mathcal{P} := \biggl\{ [\alpha] \in H^{1,1}(X,\mathbb{R}) \bigg| \int_{Y}\alpha^{\dim Y} &>0,\  \forall\, Y \subset X \mathrm{\ irreducible}\\
  &\mathrm{analytic\ subvariety},\  \dim Y>0 \biggr\}.
 \end{split}
 \end{equation*}
Then $\mathcal{K}$ is one of the connected components of $\mathcal{P}$.
\end{cor}

Next, as an application of our main Theorem \ref{main}, we answer affirmatively a question of Feldman-Ilmanen-Knopf \cite[Question 2, Section 10]{FIK} and Campana (see \cite[Section 4]{Z1})
about finite time singularities of the K\"ahler-Ricci flow on compact K\"ahler manifolds.
Let $(X, \omega_{0})$ be a compact K\"ahler manifold of dimension $n$, and consider the K\"ahler-Ricci flow with initial metric $\omega_{0}$
\begin{equation}\label{KRF}
\frac{\de}{\de t}\omega=-\Ric(\omega),\quad \omega(0)=\omega_0.
\end{equation}
The flow always has a solution for short positive time \cite{Ca}, and a result of Tian-Zhang \cite{TZ} characterizes the maximal existence time $T$ of the flow
as
$$T=\sup\{t>0\ |\ [\omega_0]-tc_1(X) \textrm{\ is\ a\ K\"ahler\ class}\}.$$
Suppose that $T<\infty$, so the flow develops a finite time singularity. The limiting class of the metrics along the flow is $[\alpha]=[\omega_0]-Tc_1(X)$, which is nef but not K\"ahler.
We will say that the K\"ahler Ricci flow does not develop a singularity at a point $x\in X$ if there is a neighborhood $U\ni x$ and a smooth K\"ahler metric $\omega_T$ on $U$ such that
$\omega(t)$ converge to $\omega_T$ in the $C^\infty$ topology as $t\to T^-$. In fact, as we will see in Theorem \ref{singul}, we can equivalenty require that the curvature of $\omega(t)$ (or even just the scalar curvature) be bounded on $U$ up to time $T$, which is a more standard definition of singularity formation in the Ricci flow (see \cite{EMT}).

The following question was asked by Feldman-Ilmanen-Knopf \cite[Question 2, Section 10]{FIK} and Campana (see \cite[Section 4]{Z1}).
\begin{que}  Do singularities of the K\"ahler-Ricci flow develop precisely along analytic subvarieties of $X$?
\end{que}

It is not hard to see that if $\int_X\alpha^n=0$ then the flow develops singularities everywhere on $X$ (see Section \ref{sectkrf}), so the answer is affirmative in this case.
We prove that this is true also when $\int_X\alpha^n>0.$

\begin{thm}\label{finitet}
Let $(X^n,\omega_0)$ be any compact K\"ahler manifold. Assume that the solution $\omega(t)$ of the K\"ahler-Ricci flow \eqref{KRF} starting at $\omega_0$
exists on the maximal time interval $[0,T)$ with $T<\infty$, and that the limiting class $[\alpha]=[\omega_0]-Tc_1(X)$ satisfies $\int_X\alpha^n>0$. Then as $t\to T^-$ the metrics $\omega(t)$ develop singularities precisely on the Zariski closed set $E_{nK}(\alpha)=\Null(\alpha)$.
\end{thm}
Note that $\Null(\alpha)$ is precisely the union of all irreducible analytic subvarieties of $X$ whose volume with respect to $\omega(t)$ approaches zero as $t\to T^-$.
So it is intuitively clear (and not hard to prove) that singularities should form at least along $\Null(\alpha)$, and our main theorem implies that no singularities develop on its complement. Note also that in this setup, Zhang \cite{Z2} proved that the Ricci curvature of $\omega(t)$ must be unbounded from below on $X$ as $t\to T^-$. Our result shows that it remains locally uniformly bounded on compact sets away from $\Null(\alpha)$. We also remark here that Tian has also asked a slightly different question in \cite[p.81]{T1}, \cite[Conjecture 2]{T2}, namely that for any finite-time singularities of the K\"ahler-Ricci flow, the weak limit as $t\to T^-$ of $\omega(t)$ as currents should be a current which is smooth at least on a Zariski open set. Theorem \ref{finitet} gives a stronger answer in the case when $\int_X\alpha^n>0$. Let us also note that in the case when $X$ is projective and $[\omega_0]$ is a rational class, the fact that no singularities develop outside $\Null(\alpha)$ follows from the work of Tian-Zhang \cite{TZ}, together with Nakamaye's result \cite{Nak}.

There are many examples of such finite-time singularities. In the case when the class $[\alpha]$ is the pullback of a K\"ahler class via the blowdown of an exceptional divisor, a delicate analysis of the singularity formation was done by Song-Weinkove \cite{SW2}. See also \cite{FIK, LT, So, So2, SSW, T1, T2, TZ, Z3} for more results about finite time singularities of the K\"ahler-Ricci flow.
We also remark that if $X$ is compact K\"ahler with nonnegative Kodaira dimension, and $T$ is {\em any} finite time singularity of the K\"ahler-Ricci flow, then we must necessarily have that $\int_X(\omega_0-Tc_1(X))^n>0$, and our Theorem \ref{finitet} applies (see Section \ref{sectkrf}).

The next application is to the study of degenerations of Ricci-flat K\"ahler metrics on Calabi-Yau manifolds.
Let $X$ be a compact K\"ahler manifold with $c_1(X)=0$ in $H^2(X,\mathbb{R})$. We will call such a manifold a Calabi-Yau manifold.
Let $\mathcal{K}\subset H^{1,1}(X,\mathbb{R})$ be the K\"ahler cone of $X$. Then Yau's theorem \cite{Y} says that for every K\"ahler class $[\alpha]\in\mathcal{K}$
there exists a unique K\"ahler metric $\omega\in[\alpha]$ with $\Ric(\omega)=0$.
Let now $[\alpha]\in \de\mathcal{K}$ be a $(1,1)$ class on the boundary of the K\"ahler cone (so $[\alpha]$ is nef) with $\int_X\alpha^n>0$.
Let $[\alpha_t]:[0,1]\to\overline{\mathcal{K}}$ be a path of $(1,1)$ classes with $[\alpha_t]\in \mathcal{K}$ for $t>0$ and with $[\alpha_0]=[\alpha]$.
For $0<t\leq 1$ write $\omega_t$ for the unique Ricci-flat K\"ahler metric in the class $[\alpha_t]$.

\begin{thm}\label{main1}Let $X^n$ be a Calabi-Yau manifold and $[\alpha]\in\de \mathcal{K}$ be a nef class with $\int_X\alpha^n>0$.
Then there is a smooth incomplete Ricci-flat K\"ahler metric $\omega_0$ on $X\backslash \Null(\alpha)$ such that given any path $[\alpha_t]$, $0<t\leq 1$, of K\"ahler classes approaching $[\alpha]$ as $t\to 0$, the Ricci-flat K\"ahler metrics $\omega_t$ in the class $[\alpha_t]$ converge to $\omega_0$ as $t\to 0$ in
the $C^\infty_{\mathrm{loc}}(X\backslash \Null(\alpha))$ topology. Moreover, $(X,\omega_t)$ converges as $t\to 0$ to the metric completion of $(X\backslash \Null(\alpha),\omega_0)$ in the Gromov-Hausdorff topology.
\end{thm}
The same result holds if instead of a path one considers a sequence of K\"ahler classes $[\alpha_i]$ converging to $[\alpha]$ (the proof is identical).
In the special case when $X$ is projective and the class $[\alpha]$ belongs to the real N\'eron-Severi group (i.e. it is the first Chern class of an $\mathbb{R}$-divisor), this was proved by the second author in \cite{deg}, except for the statement about Gromov-Hausdorff convergence which is due to Rong-Zhang \cite{RZ} (see also \cite{RZ2, So2}). In this case there is a birational contraction map $\pi:X\to Y$ to a singular Calabi-Yau variety, which is an isomorphism away from $\Null(\alpha)$, and such that $\omega_0$ is the pullback of a singular K\"ahler-Einstein metric on $Y$ \cite{EGZ}. The contraction map is also used crucially to construct a smooth semipositive representative of the limiting class $[\alpha]$. The interested reader can also consult the surveys \cite{surv, nyac}.
In our more general setup there is no contraction map, but our theorem still shows that set where the metrics $\omega_t$ fail to converge smoothly to a limiting K\"ahler metric is precisely equal to $\Null(\alpha)$, exactly as in the projective case (see Section \ref{sectcy}).

The main technical ingredient in the proof of our main theorem is a technical extension-type result for K\"ahler currents with analytic singularities in a nef and big class which are defined on submanifold of $X$ (Theorem \ref{extend}).
The problem of extending positive currents from subvarieties has recently generated a great deal of interest in both analytic and geometric applications \cite{CGZ, His, Sch, Wu}.  At the moment, the best results available are due to Coman-Guedj-Zeriahi \cite{CGZ}, who proved that every positive current in a Hodge class on a (possibly singular) subvariety $V \subset \mathbb{CP}^{N}$ is the restriction of a globally defined, positive current.  Their proof uses a result of Coltoiu \cite{Co} on extending Runge subsets of analytic subsets of $\mathbb{C}^{N}$ and a delicate, inductive gluing procedure. By contrast, the technique developed here is constructive, and relies only on a gluing property of plurisubharmonic functions with analytic singularities.  The main technical tool is the use of resolution of singularities in order to obtain estimates which allow us to glue plurisubharmonic functions near their singularities.  The ideas used here are similar, and indeed inspired by, techniques developed by Collins-Greenleaf-Pramanik \cite{CGP} to obtain sharp estimates for singular integral operators. In turn, these were motivated by the seminal work of Phong, Stein and Sturm \cite{PSS, PS, PS2}. The advantage of these local techniques is that we can work on general K\"ahler manifolds and with arbitrary K\"ahler classes. In our paper \cite{CT}, we refined these techniques and proved an extension theorem for K\"ahler currents with analytic singularities in K\"ahler classes on complex submanifolds.

This paper is organized as follows. In Section \ref{sectprel} we give some background and definitions, and show how Corollaries \ref{locus} and \ref{demap} follow from Theorem \ref{main}.
The main Theorem \ref{main} is proved in Section \ref{sectproof}, where we also establish an extension-type result for K\"ahler currents (Theorem \ref{extend}). Section \ref{sectkrf} is devoted to the study of finite time singularities of the K\"ahler-Ricci flow, and contains the proof of Theorem \ref{finitet}. Finally, Theorem \ref{main1} on degenerations of Calabi-Yau metrics is proved in Section \ref{sectcy}.\\

{\bf Acknowledgments. }We would like to thank D.H. Phong for all his advice and support,
and S. Boucksom, S. Dinew, G. Sz\'ekelyhidi, B. Weinkove, S. Zelditch, A. Zeriahi, Y. Zhang and Z. Zhang for comments and discussions. We also thank the referees for carefully reading our manuscript, for helpful comments and corrections, and for a useful suggestion which simplified our original proof of Theorem \ref{extend}.

\section{Preliminaries}\label{sectprel}
In this section we give the necessary definitions and background, and prove Corollaries \ref{locus} and \ref{demap} assuming our main Theorem \ref{main}.

Throughout this section and the next one, $X$ is an $n$-dimensional compact complex manifold in Fujiki's class $\mathcal{C}$ \cite{Fu} (i.e. bimeromorphic to a compact K\"ahler manifold), and $\omega$ is a Hermitian metric on $X$. We will denote by $H^{1,1}(X,\mathbb{R})$ the Bott-Chern cohomology of $d$-closed real $(1,1)$ forms (or currents) modulo $\de\dbar$-exact ones.

Let $\alpha$ be a smooth closed real $(1,1)$-form on $X$.
We will say that a function $\phi$ on $X$ is $\alpha$-plurisubharmonic ($\alpha$-PSH) if $\phi$ is in $L^1(X)$, upper semicontinuous, and $\alpha+\ddb \phi\geq 0$ as currents.
In this case $T=\alpha+\ddb\phi$ is a closed positive $(1,1)$ current on $X$, in the cohomology class $[\alpha]\in H^{1,1}(X,\mathbb{R})$, and conversely every closed positive
$(1,1)$ current in the class $[\alpha]$ is of this form. For details about this, and other basic facts on closed positive currents, we refer the reader to the book of Demailly \cite{Demb}.

Closed positive currents are in general singular, and a convenient way to measure their singularities is the Lelong number
$\nu(T,x).$ We can also define the Lelong sublevel sets
$$E_c(T)=\{x\in X\ |\ \nu(T,x)\geq c\},$$
where $c>0$, and
\begin{equation*}
E_+(T) := \bigcup_{c>0}E_c(T).
\end{equation*}
A fundamental result of Siu \cite{Siu74} shows that the Lelong sublevel sets $E_c(T)$ are in fact closed, analytic subvarieties of $X$.  It follows that $E_+(T)$ is a countable union of analytic subvarieties. We also define a K\"ahler current to be a closed positive $(1,1)$ current $T$ such that $T\geq \ve\omega$ holds as currents on $X$ for some $\ve>0$. Thanks to a result of Demailly-P\u{a}un \cite{DP}, a compact complex manifold is in class $\mathcal{C}$ if and only if it admits a K\"ahler current.
If a $(1,1)$ class $[\alpha]$ contains a K\"ahler current, it is called big. This terminology comes from algebraic geometry: if $X$ is projective and $L$ is a holomorphic line bundle on $X$, then $L$ is big (in the sense that $h^0(X,L^k)\sim k^n$ for $k$ large) if and only if $c_1(L)$ is big in the above sense \cite{Dem92b}.

If $X$ is K\"ahler, the set of all K\"ahler classes on $X$ is the K\"ahler cone $\mathcal{K}\subset H^{1,1}(X,\mathbb{R})$. Classes in its closure $\ov{\mathcal{K}}$ are called nef (or numerically effective). A class $[\alpha]$ is nef if and only if for every $\ve>0$ there exists a smooth function $\rho_\ve$ on $X$ such that
$\alpha+\ddb\rho_\ve \geq -\ve\omega.$ One then takes this as the definition of a nef class on a general compact complex manifold.
It is easy to see that the pullback of a nef class by a holomorphic map between compact complex manifolds is still nef.
Moreover, it is clear that if the class $c_{1}(L)$ is nef in the above sense then the line bundle $L$ is nef in the usual algebro-geometric sense (i.e. $L\cdot C\geq 0$ for all curves $C\subset X$), and the converse is true when $X$ is projective \cite{Dem92b} and also when $X$ is Moishezon thanks to a result of P\u{a}un \cite{Paun}, but false for general compact K\"ahler manifolds (a generic complex torus $\mathbb{C}^n/\Lambda$ provides many counterexamples).

A useful notion, introduced by Demailly \cite{Dem92}, is that of currents with analytic singularities. These are closed positive $(1,1)$ currents $T=\alpha+\ddb\phi$ such that
given each point $x\in X$ there is an open neighborhood $U$ of $x$ with holomorphic functions $f_1,\dots,f_N$ such
that on $U$ we have
\begin{equation}\label{anal}
\vp=\delta\log\left(\sum_j |f_j|^2\right)+g,
\end{equation}
where $g$ is a smooth function on $U$ and $\delta$ is a nonnegative real number. The holomorphic functions $f_j$ generate a coherent ideal sheaf which defines the analytic singularities of $\vp$. In this case the current $T$ is smooth away from an analytic subvariety of $X$, which in this case is exactly $E_+(T)$ (and also equal to $E_c(T)$ for some $c>0$). Then Demailly's regularization procedure \cite{Dem92} shows that if a class $[\alpha]$ is big, then it also contains a K\"ahler current with analytic singularities. We will also say that a positive $(1,1)$ current $T$ has {\em weakly analytic singularities} if its potentials are locally of the form \eqref{anal} with $g$ continuous.

A fundamental theorem of Demailly-P\u{a}un \cite[Theorem 2.12]{DP} states that if $X$ is a compact K\"ahler manifold, and $[\alpha]$ is nef with $\int_X\alpha^{n}>0$, then $[\alpha]$ contains a K\"ahler current (i.e. $[\alpha]$ is big).

Finally, we introduce a number of different subsets of $X$, all of which aim to measure the defect from a class $[\alpha]$ being K\"ahler.  In the rational case, the sets measure the defect from the corresponding $\mathbb{Q}$-divisor being ample.

\begin{defn}\label{positivity}
Let $(X,\omega)$ be an $n$-dimensional compact complex manifold in class $\mathcal{C}$ equipped with a Hermitian metric, and let $[\alpha] \in H^{1,1}(X,\mathbb{R})$ be a big class.
\begin{enumerate}
\item[(i)] Following \cite{Bou}, we define the \emph{non-K\"ahler locus}, denoted $E_{nK}(\alpha)$, to be the set
\begin{equation*}
E_{nK}(\alpha) := \bigcap_{T \in [\alpha]} E_{+}(T),
\end{equation*}
where the intersection is taken over all K\"ahler currents in the class $[\alpha]$.
\item[(ii)] We define the \emph{null locus}, denoted by $\Null(\alpha)$, to be the set
\begin{equation*}
\Null(\alpha) := \bigcup_{\int_{V}\alpha^{\dim V}=0}V
\end{equation*}
where the union is taken over all positive dimensional irreducible analytic subvarieties of $X$.
\end{enumerate}
\begin{enumerate}
\item[(iii)] Following \cite{BEGZ}, we define the \emph{ample locus}, denoted by $\amp(\alpha)$ to be the set of points $x \in X$ such that there exists a K\"ahler current $T \in [\alpha]$ with analytic singularities, which is smooth in a neighborhood of $x$.
\end{enumerate}
Finally, suppose $X$ is a smooth projective variety, and
$$[\alpha] \in NS(X,\mathbb{R})=(H^2(X,\mathbb{Z})_{\mathrm{free}}\cap H^{1,1}(X))\otimes \mathbb{R},$$
belongs to the real N\'eron-Severi group, so $[\alpha]=c_1(D)$ for $D$ an $\mathbb{R}$-divisor:
\begin{enumerate}
\item[(iv)] We define the \emph{augmented base locus}, denoted $\mathbb{B}_{+}(\alpha)$, or $\mathbb{B}_{+}(D)$ to be the stable base locus of $D-A$ for some small, ample $\mathbb{R}$-divisor $A$ with $D-A$ a $\mathbb{Q}$-divisor.  That is, $\mathbb{B}_{+}(\alpha)$ is the base locus of the linear series $|m(D-A)|$ for $m$ sufficiently large and divisible.
\end{enumerate}
\end{defn}
The definition of $\mathbb{B}_{+}(D)$ in (iv) is well-posed, since we get the same stable base locus for all sufficiently small ample $\mathbb{R}$-divisors $A$ with $D-A$ a $\mathbb{Q}$-divisor \cite[Proposition 1.5]{ELMNP2}, and $\mathbb{B}_{+}(D)$ depends only on the numerical class of $D$ (so the notation $\mathbb{B}_{+}(\alpha)$ is justified). Furthermore, $\mathbb{B}_+(D)$ equals the intersection of the supports of $E$ for all ``Kodaira's Lemma''-type decompositions $D=A+E$, where $A$ is an ample $\mathbb{R}$-divisor and $E$ is an effective $\mathbb{R}$-divisor \cite[Remark 1.3]{ELMNP2}.

In \cite[Theorem 3.17 (ii)]{Bou}, Boucksom observed that in a big class we can always find a K\"ahler current with analytic singularities precisely along
$E_{nK}(\alpha)$. Combining this statement with the theorem of Demailly-P\u{a}un \cite{DP} and Demailly's regularization \cite{Dem92}, we get the following:

\begin{thm}\label{usefull}
Let $X^n$ be a compact complex manifold in class $\mathcal{C}$ and $[\alpha] \in H^{1,1}(X,\mathbb{R})$ be a nef class with $\int_X\alpha^n>0$.
Then there exists a K\"ahler current $T$ in the class $[\alpha]$, with analytic singularities such that
\begin{equation*}
E_{+}(T) = E_{nK}(\alpha).
\end{equation*}
In particular $E_{nK}(\alpha)$ is an analytic subvariety
 of $X$.
\end{thm}

Indeed, since $X$ is in class $\mathcal{C}$ there is a modification $\mu:\ti{X}\to X$ with $\ti{X}$ a compact K\"ahler manifold \cite{Fu, Va}. The class $[\mu^*\alpha]$ is then nef and satisfies
$\int_{\ti{X}}(\mu^*\alpha)^n=\int_X\alpha^n>0$, and so it contains a K\"ahler current $T$ by Demailly-P\u{a}un \cite{DP}. It is then easy to check that the pushforward $\mu_*T$ is a K\"ahler current on $X$ in the class $[\alpha]$. Then Demailly's regularization \cite{Dem92} and Boucksom's argument \cite[Theorem 3.17 (ii)]{Bou} (neither of which needs $X$ to be K\"ahler) complete the proof.

Using this result, we can show how Corollary \ref{demap} follows from Theorem \ref{main}.
\begin{proof}[Proof of Corollary \ref{demap}]
The inclusion $\mathcal{K} \subset \mathcal{P}$ is clear.  Moreover, $\mathcal{K}$ is clearly connected and open.  Suppose $[\alpha] \in \overline{\mathcal{K}}\cap \mathcal{P}$.  Then $[\alpha]$ is nef and $\int_X\alpha^n>0$.  Moreover, since $[\alpha] \in \mathcal{P}$, we have $\Null(\alpha) = \emptyset$.  By Theorem \ref{main} we have that $E_{nK}(\alpha) =\emptyset$.  Thanks to Theorem \ref{usefull} there exists $T \in[\alpha]$ a K\"ahler current with analytic singularities and with $E_{+}(T) = \emptyset$, which means that $T$ is a smooth K\"ahler metric.  Thus, $[\alpha] \in \mathcal{K}$, and so $\mathcal{K}$ is closed in $\mathcal{P}$.  Since $\mathcal{K}$ is connected, the corollary follows.
\end{proof}

There are some relations among the subsets of $X$ introduced in Definition \ref{positivity}. The following two propositions are
available in the literature, and we include proofs for the convenience of the reader.

\begin{prop}[Boucksom \cite{Bou}]\label{nkamp}
Let $X^n$ be a compact complex manifold in class $\mathcal{C}$ and $[\alpha] \in H^{1,1}(X,\mathbb{R})$ be a nef class with $\int_X\alpha^n>0$. Then $E_{nK}(\alpha) = \amp(\alpha)^{c}$.
\end{prop}
\begin{proof}
Apply Theorem \ref{usefull}, to get a K\"ahler current $T\in[\alpha]$ with analytic singularities
with $E_+(T)=E_{nK}(\alpha)$. Since $T$ has analytic singularities, it is smooth precisely wherever its Lelong number vanishes, so $T$ is smooth off $E_+(T)$. This shows that $\amp(\alpha)^{c}\subset E_{nK}(\alpha)$. The reverse inclusion is trivial.
\end{proof}

Corollary \ref{locus} follows immediately from Theorem \ref{main} and the following:
\begin{prop}[Boucksom \cite{BoT}, Corollary 2.2.8]
If $[\alpha]$ is the class of a big and nef $\mathbb{R}$-divisor $D$ on a smooth projective variety $X$ over $\mathbb{C}$, then $E_{nK}(\alpha) = \mathbb{B}_{+}(\alpha)$.
\end{prop}

\begin{proof}
For the inclusion $\mathbb{B}_{+}(\alpha) \subset E_{nK}(\alpha)$ we essentially follow \cite[Proposition 5.2]{Bou}. If $x\not\in E_{nK}(\alpha)$, then there exists a K\"ahler current $T$ in the class $[\alpha]$ with analytic singularities and smooth in a coordinate neighborhood $U$ of $x$. Choose $A$ an ample divisor and $\delta'>0$ a small constant such that
 $D-\delta' A$ is a $\mathbb{Q}$-divisor. Fix $\omega$ a K\"ahler metric in $c_1(A)$, and choose $0<\delta<\delta'$ small enough so that $T-\delta\omega$ is still a K\"ahler current and $\delta'-\delta\in\mathbb{Q}$, so that $D-\delta A=D-\delta'A+(\delta'-\delta)A$ is a $\mathbb{Q}$-divisor.
Let $\theta$ be a smooth cutoff function supported in $U$ and identically $1$ near $x$, and let
$$\ti{T}=T-\delta\omega+\ve\ddb (\theta\log|z-x|),$$
where $\ve>0$ is small enough so that $\ti{T}$ is a K\"ahler current in $c_1(D-\delta A)$. The Lelong number of $\ti{T}$ at $x$ is $\ve$, and $\ti{T}$ is smooth on $U\backslash\{x\}$. Let $\beta$ be a smooth representative of $-c_1(X)$, and let
$\ti{T}_m=m\ti{T}-\beta$. For $m$ sufficiently large, $\ti{T}_m$ is a K\"ahler current in $c_1(m(D-\delta A)-K_X)$, smooth on $U\backslash\{x\}$ and with Lelong number $m\ve$ at $x$.
If $m$ is sufficiently divisible, then $m(D-\delta A)$ corresponds to a holomorphic line bundle $L_m$. From the definition of Lelong number, it follows easily (cf. Skoda's lemma \cite[Lemma 5.1]{Bou}) that if $m$ is large the subvariety cut out by the multiplier ideal sheaf $\mathcal{I}(\ti{T}_m)$ of $\ti{T}_m$ intersects $U$ in the isolated point $x$. Nadel's vanishing theorem then implies that
$H^1(X, L_m\otimes\mathcal{I}(\ti{T}_m))=0$, and therefore the restriction map
$H^0(X, L_m)\to (L_m)_x$ is surjective. Hence there is a global section of $L_m$ that does not vanish at $x$. This means that $x\in \mathbb{B}_{+}(\alpha)^c$.

To prove the reverse inclusion, take a point $x\in \mathbb{B}_+(\alpha)^c$, so by definition there is an ample $\mathbb{R}$-divisor $A$
such that the stable base locus of $D-A$ does not contain $x$. This means that there is a large integer $m$ such that $mD-mA=c_1(L)$ for some line bundle $L$ and there is an effective $\mathbb{Z}$-divisor $E$ linearly equivalent to $mD-mA$ which does not pass through $x$. Take $s\in H^0(X,L)$ a defining section for $E$, and complete it to a basis $\{s=s_1,s_2,\dots,s_N\}$ of $H^0(X,L)$. Fixing a smooth Hermitian metric $h$ on $L$ with curvature form $R(h)$ we can define a closed positive current $T_L=R(h)+\ddb\log\sum_i |s_i|^2_h\in c_1(L)$ which has analytic singularities and is smooth near $x$. If $\omega$ is a K\"ahler metric in $c_1(A)$, then $\frac{1}{m}T_L+\omega$ is then a K\"ahler current on $X$ in the class $[\alpha]$ with analytic singularities and smooth near $x$, so $x\in \amp(\alpha)^c$, which equals $E_{nK}(\alpha)^c$ by Proposition \ref{nkamp}.
\end{proof}

Let us discuss the structure of $\Null(\alpha)$ when $n=2$, so $X$ is a complex surface and $[\alpha]$ is a nef class with positive self-intersection which is assumed not to be K\"ahler. It is well-known that complex surfaces in class $\mathcal{C}$ are K\"ahler \cite{BHPV}.
In this case, $\Null(\alpha)$ is the union of a family of irreducible curves $C_i$, $1\leq i\leq N$, with the intersection matrix $(C_i\cdot C_j)$ negative definite. Indeed for any real numbers $\lambda_i$ we have $\int_X\alpha^2>0$ and $\int_{\sum_i \lambda_i C_i}\alpha=0$, so by the Hodge index theorem \cite[Corollary IV.2.16]{BHPV} either $\sum_i \lambda_i C_i=0$ or $(\sum_i \lambda_i C_i)^2<0$. This means that the intersection form is negative definite on the linear span of the curves $C_i$, which proves our assertion. Grauert's criterion \cite[Theorem III.2.1]{BHPV} shows that each connected component of $\Null(\alpha)$ can be contracted, so we get a map $\pi:X\to Y$ where $Y$ is a normal complex surface. In general though, $Y$ is not a K\"ahler space.

We conclude this section by proving the easy inclusion in our main Theorem \ref{main}. An alternative proof of this result (in the case when $X$ is K\"ahler) can be found in \cite[Theorem 2.2]{CZ}.
\begin{thm}\label{contain}
Let $X$ be a compact complex manifold in class $\mathcal{C}$, let $\alpha$ be a closed smooth real $(1,1)$-form on $X$ whose class $[\alpha]\in H^{1,1}(X,\mathbb{R})$ is nef with $\int_X \alpha^n>0$. Then
\begin{equation}\label{toprove}
\Null(\alpha)\subset E_{nK}(\alpha).
\end{equation}
\end{thm}
\begin{proof}
Fix a point $x\in X$ and an irreducible analytic subvariety $V\subset X$ of dimension $k$, such that $x \in V$, and $\int_{V} \alpha^k =0$.  For the sake of obtaining a contradiction, suppose also that $x \in \amp(\alpha)$ so there exists a K\"ahler current $T \in [\alpha]$ with analytic singularities, smooth in a neighborhood of $x$.
Since $V$ is a subvariety of $X$ which is in $\mathcal{C}$, it follows from \cite{Fu} that $V$ is also in class $\mathcal{C}$, so there is a modification $\mu:\ti{V}\to V$ with $\ti{V}$ a compact K\"ahler manifold of dimension $k$. The class $[\mu^*\alpha]$ is nef and satisfies $\int_{\ti{V}}(\mu^*\alpha)^k=0$, and it contains the closed positive $(1,1)$ current $\mu^*T$ (defined as usual by writing $T=\alpha+\ddb \phi$ and letting $\mu^*T=\mu^*\alpha+\ddb(\phi\circ\mu)$). Furthermore $\mu^*T$ is a smooth K\"ahler metric on a nonempty open set $\ti{U}\subset\ti{V}$ .
By \cite[Proposition 1.6]{BEGZ} the non-pluripolar product $\langle \mu^*T\rangle^k$ is a well-defined closed positive $(k,k)$-current on $\ti{V}$. By \cite[Proposition 1.20]{BEGZ} we have that
$$0\leq \int_{\ti{V}}\langle\mu^*T\rangle^k\leq \int_{\ti{V}}(\mu^*\alpha)^k=0,$$
so $\int_{\ti{V}}\langle\mu^*T\rangle^k=0$. But $\mu^*T$ is a smooth K\"ahler metric on $\ti{U}$, and so
$$0<\int_{\ti{U}} (\mu^*T)^k=\int_{\ti{U}} \langle \mu^*T\rangle^k\leq \int_{\ti{V}}\langle\mu^*T\rangle^k,$$
which is a contradiction.
\end{proof}

\section{Proof of Theorem \ref{main}}\label{sectproof}
In this section we give the proof of Theorem \ref{main}.
The idea is to observe that if there is a point $x \in \Null(\alpha)^{c} \cap E_{nK}(\alpha)$, then there is
an entire analytic subvariety $V\subset E_{nK}(\alpha)$ of positive dimension, for which the restricted class is again nef
and with positive self-intersection.  Moreover, there is a K\"ahler current $R$ in the class $[\alpha]$ with analytic singularities
along $E_{nK}(\alpha)$ which by assumption has positive Lelong numbers at every point of $V$.  Assume for the moment
that $V$ is smooth.  Using the mass concentration technique of Demailly-P\u{a}un \cite{DP}, it follows that $V$
supports a K\"ahler current $T$ in the class $[\alpha|_V]$ with analytic singularities.  If we can use $T$ to produce a K\"ahler current $\ti{T}$ in the class $[\alpha]$ defined on all of $X$ and $\ti{T}|_V$ is smooth near the generic point of $V$, then the theorem follows, since the generic point of $V$
is then contained in $E_{nK}(\alpha)^{c}$, contradicting the definition of $V$. The case
of $V$ singular is similar after passing to a resolution of singularities.

We start with the following lemma, which is the analytic counterpart of \cite[Proposition 1.1]{ELMNP} in the algebraic case.
Note that by definition $\Null(\alpha)$ does not have any isolated points.
\begin{lem}\label{isolated}
Let $X$ be a compact complex manifold in class $\mathcal{C}$ and $[\alpha]$ a real $(1,1)$ class which is nef and with positive self-intersection. Then the analytic set
$E_{nK}(\alpha)$ does not have any isolated points.
\end{lem}
\begin{proof}
If $x\in E_{nK}(\alpha)$ is an isolated point, then we can find a small open set
 $U \subset X$ which contains the point $x$, such that $U\cap E_{nK}(\alpha)=\{x\}$.
By shrinking $U$ if necessary, we can assume that $U$ has coordinates $z=(z_{1},\dots,z_{n})$.
If $A$ is a sufficiently large constant, then on $U$ we have
 \begin{equation*}
 \alpha +\ddb(A|z|^{2}-C) \geq \omega,
 \end{equation*}
for any constant $C$.
Let $R = \alpha + \ddb F$ be a K\"ahler current in $[\alpha] \in H^{1,1}(X,\mathbb{R})$ which has analytic singularities, which exists thanks to
a theorem of Demailly-P\u{a}un \cite{DP} and Demailly's regularization \cite{Dem92}. Furthermore, we can choose $R$ with the property that
$E_+(R)=E_{nK}(\alpha)$, thanks to Theorem \ref{usefull}, and $R$ is smooth outside of $E_{+}(R)$.
In particular $F$ is smooth near $\de U$, and so we can choose $C$ sufficiently large such that
$\max\{F, A|z|^{2}-C\} = F$ near $\del U$.
Fix a small $\ve>0$ and let $\max_{\ve}$ be the regularized maximum function (see, e.g. \cite[I.5.18]{Demb}).
Then $\ti{F}=\max_{\ve}(F, A|z|^{2}-C)$ equals $F$ near $\de U$ and equals $A|z|^2-C$ near $x$, is smooth on $U$,
and so $\ti{R}=\alpha+\ddb \ti{F}$ is a K\"ahler
current on $X$ in the class $[\alpha]$ with analytic singularities and with
$x \notin E_{+}(\ti{R})$, from which it follows that $x \notin E_{nK}(\alpha)$,
a contradiction.
\end{proof}

\begin{proof}[Proof of Theorem \ref{main}]
Thanks to Theorem \ref{contain}, we only need to show that $E_{nK}(\alpha)\subset \Null(\alpha)$.
For the sake of obtaining a contradiction, assume the result does not hold, and there exists $x \in \Null(\alpha)^{c} \cap E_{nK}(\alpha)$.
Let $V$ be an irreducible component of the analytic set $E_{nK}(\alpha)$ passing through $x$.
Here and in the following whenever we deal with analytic sets, we always consider the underlying reduced variety.
From Lemma \ref{isolated}, we must have $\dim V>0$.
Thus, we have that $V \subset E_{nK}(\alpha)$ is a positive dimensional irreducible analytic subvariety, with $x \in V$.
Moreover, since $x\in \Null(\alpha)^{c}$, we have that $[\alpha|_{V}]$ has positive self-intersection, and if $V$ is smooth then
$[\alpha|_{V}]$ is also nef (this is true also for $V$ singular, provided one properly defines nef classes on singular varieties, which we will not need to do). We will need the following theorem:

\begin{thm}\label{extend}
Let $X$ be a compact complex manifold in class $\mathcal{C}$ and $\alpha$ a closed smooth real $(1,1)$ form on $X$ with $[\alpha]$ nef and $\int_X\alpha^n>0$.
Let $E=V\cup\bigcup_{i=1}^I Y_i$ be an analytic subvariety of $X$, with $V, Y_i$ its irreducible components, and $V$ a positive dimensional compact complex submanifold of $X$. Let $R=\alpha+\ddb F$ be a K\"ahler current in the class $[\alpha]$ on $X$ with analytic singularities precisely along $E$ and let $T=\alpha|_{V} +\ddb\phi$ be a K\"ahler current in the class $[\alpha|_V]$ on $V$ with analytic singularities.
Then there exists a K\"ahler current $\ti{T}=\alpha+\ddb\Phi$ in the class $[\alpha]$ on $X$ with $\ti{T}|_V$ smooth in a neighborhood of the generic point of $V$.
\end{thm}

Let us complete the proof of Theorem \ref{main} assuming this result.
Assume first that $V$ is smooth. Then by Demailly-P\u{a}un \cite{DP} (see Theorem \ref{usefull}) there exists a K\"ahler current $T=\alpha|_V+\ddb\phi \in [\alpha|_{V}]$. Furthermore thanks to Demailly's regularization \cite{Dem92}, we can choose $T$ to have analytic singularities, so in particular $T$ is smooth at the generic point of $V$.
We also write
\begin{equation*}
E_{nK}(\alpha) = V \cup\bigcup_{i=1}^I Y_{i},
\end{equation*}
for distinct irreducible components $Y_{i}$ of $E_{nK}(\alpha)$, which are all positive dimensional thanks to Lemma \ref{isolated}.

We then apply Theorem \ref{extend} to $E=E_{nK}(\alpha)$, with the current $R$ constructed as in the proof of Lemma \ref{isolated}, and we conclude that
there is a K\"ahler current  $\ti{T}=\alpha+\ddb\Phi$ in the class $[\alpha]$ on $X$ with $\ti{T}|_V$ smooth in a neighborhood of the generic point of $V$. In particular, for a generic $x\in V$ we have
$\nu(\ti{T}|_V,x)=0$. But we always have that the Lelong number of $\ti{T}$ at $x$ calculated in the ambient space $X$ satisfies $\nu(\ti{T},x)\leq \nu(\ti{T}|_V,x)$, hence $\ti{T}$ has Lelong number zero at a generic point of $V$, and so $V\nsubseteq E_{+}(\ti{T})$.
This immediately leads to a contradiction in this case, because by definition
\begin{equation*}
V\subset  E_{nK}(\alpha)\subset E_{+}(\ti{T}).
\end{equation*}

If $V$ is not smooth we take $\pi:\tilde{X} \rightarrow X$ be an embedded resolution of singularities
of $V$, obtained by blowing up smooth centers (which exists thanks to work of Hironaka \cite{Hi}), and let $\tilde{V}$ be the proper transform of $V$, which is now
a smooth submanifold of $\tilde{X}$. The manifold $\ti{X}$ is in class $\mathcal{C}$, and $\pi:\tilde{V} \rightarrow V$ is an isomorphism over $V_{reg}$.
Note that since $\pi$ is bimeromorphic, we have $\pi_*\pi^*\alpha=\alpha$ as currents. Therefore, if $T$ is a K\"ahler current in the class $[\pi^*\alpha]$ on $\ti{X}$ then
its pushforward $\pi_*T$ is a K\"ahler current in the class $[\alpha]$ on $X$.

The class $[(\pi^{*}\alpha)|_{\ti{V}}]$ is nef and has positive self-intersection on $\tilde{V}$, and hence by Demailly-P\u{a}un \cite[Theorem 2.12]{DP} and Demailly's regularization
there exists a K\"ahler current $T \in [(\pi^{*}\alpha)|_{\ti{V}}]$ with analytic singularities. Furthermore, the class $[\pi^*\alpha]$ is nef and has positive self-intersection on $\ti{X}$,
so by the same reason there exists a K\"ahler current $R\in [\pi^*\alpha]$ with analytic singularities along an analytic subset $E$ of $\ti{X}$.
We must have that $\ti{V}\subset E$, because if not then $R$ is smooth near the generic point of $\ti{V}$, and
then the pushforward $\pi_*R$ would be a K\"ahler current in the class $[\alpha]$ on $X$ which is smooth at the
generic point of $V$, which contradicts $V\subset E_{nK}(\alpha)$.

We can then apply Theorem \ref{extend} and get a K\"ahler current $\ti{T}$ on $\tilde{X}$ in the class $[\pi^*\alpha]$
with $\ti{T}|_{\ti{V}}$ smooth near the generic point of $\tilde{V}$.
Therefore $\pi_{*}\ti{T}$ is then a K\"ahler current on $X$ in the class $[\alpha]$ whose restriction to $V$ is smooth near the
generic point of $V$ (note that, in general, it may be the case that $\pi_{*}\ti{T}$ has non-vanishing Lelong numbers along $V_{sing}$,
which is however a proper subvariety of $V$). Again, we have a contradiction, which proves our main theorem.
\end{proof}

We now turn to the proof of Theorem \ref{extend}.
The proof is somewhat technical, so we briefly outline the main steps.  First, we perform a resolution of singularities to replace the singularities of $T$ with a divisor with simple normal crossings on the proper transform $\ti{V}$ of $V$, which is the restriction of a simple normal crossings divisor on the ambient space. We then construct local extensions on a finite open covering of $\ti{V}$, which have the same singularities on any overlap. This is a crucial point, which allows us to
use a technique of Richberg \cite{Rich} to glue the local potentials to construct a K\"ahler current in a neighborhood of $\ti{V}$. Lastly, after possibly reducing the Lelong numbers
of the global K\"ahler current $R$, we can glue the K\"ahler current on this neighborhood to $R$, to obtain a global K\"ahler current.

\begin{proof}[Proof of Theorem \ref{extend}]
Let $\mathcal{I}_T\subset\mathcal{O}_V$ be the ideal sheaf defining the analytic singularities of the K\"ahler current $T$ on $V$.
Thanks to the fundamental work of Hironaka \cite{Hi} there is a principalization $\pi_1:\tilde{V}_1 \rightarrow V$ of $\mathcal{I}_T$ which is obtained by a finite sequence of blowups along smooth centers. The center of the first blowup $p_1:V_1\to V$ is a submanifold $Z$ of $V$, hence a submanifold of $X$, and the blowup $q_1:X_1\to X$ of $X$ along $Z$ has the property that the proper transform $\ti{V}_1$ of $V$ is smooth and isomorphic to $V_1$ and the restriction of $q_1$ to $\ti{V}_1$ composed with this isomorphism equals $p_1$. Furthermore, the exceptional divisor $E$ of $q_1$ is in normal crossings with $\ti{V}_1$ and $E\cap \ti{V}_1$ is identified with the exceptional divisor of $p_1$. We can then repeat this procedure, and obtain a birational morphism $\pi:\ti{X}\to X$ (with $\ti{X}$ smooth), which is a composition of blowups with smooth centers, such that the proper transform $\ti{V}$ of $V$ is smooth and
\begin{equation}\label{ideal}
(\pi|_{\ti{V}})^{*}\mathcal{I}_T = \mathcal{O}_{\ti{V}}\left(-\sum_{\ell}a_\ell D_{\ell}|_{\ti{V}}\right),
\end{equation}
where $a_\ell>0$, where $\cup_\ell D_\ell$ is a union of smooth divisors with simple normal crossings in $\ti{X}$, which furthermore has normal crossings with $\tilde{V}$.
This means that $\ti{V}$ is a smooth submanifold of $\ti{X}$, which can be covered by finitely many charts $\{W_j\}_{1\leq j\leq N}$ such that on each $W_j$ there are local coordinates
$(z_1,\dots,z_n)$ with $\ti{V}\cap W_j=\{z_1=\dots=z_{n-k}=0\}$, where $k=\dim V$, and with $\cup_\ell D_\ell\cap W_j=\{z_{i_1}\cdots z_{i_p}=0\}, n-k<i_1,\dots,i_p\leq n$ (see e.g. \cite[Theorem 2.0.2]{Wl}, \cite{Hi}).
Write $z=(z_1,\dots,z_{n-k})$ and $z'=(z_{n-k+1},\dots,z_n)$. Since the variety cut out by $\mathcal{I}_T$ has codimension at least $2$ in $X$, we have that all divisors $D_\ell$ are $\pi$-exceptional, and we may also assume
that these are all the $\pi$-exceptional divisors.

Choose $\ve>0$ small enough so that $$T= \alpha|_{V} +\ddb\phi \geq 3\epsilon \omega|_{V},$$ holds as currents on $V$. Pulling back to $\ti{X}$ we obtain
$$\pi^*\alpha|_{\ti{V}}+\ddb (\vp\circ\pi)\geq 3\ve\pi^*\omega|_{\ti{V}}.$$
Let $s_D$ be a defining section of $\mathcal{O}_{\ti{X}}(\sum_\ell D_\ell)$.
The smooth form $\pi^*\omega$ is not positive definite on $\ti{X}$, but there is a small $\delta>0$ such that
$$\pi^*\omega+\delta\ddb\log|s_D|^2_{h_D}\geq \ti{\omega},$$
for some Hermitian metric $\ti{\omega}$ on $\ti{X}$, where $h_D$ is a suitable smooth metric on $\mathcal{O}_{\ti{X}}(\sum_\ell D_\ell)$ (see e.g. \cite[Lemma 6]{PS3} or \cite[Proposition 3.24]{Vo}).
Then
$$\pi^*\alpha|_{\ti{V}}+\ddb(\vp\circ\pi)+\ddb(3\ve\delta \log|s_D|^2_{h_D})|_{\ti{V}}\geq 3\ve\ti{\omega}|_{\ti{V}}.$$
For simplicity of notation, define $\hat{\vp}=(\vp\circ\pi)+3\ve\delta \log|s_D|^2_{h_D},$ which is a function on $\ti{V}$.
Define a function $\vp_j$ on $W_j$ (with analytic singularities) by
\begin{equation}\label{ideal2}
\vp_j(z,z')=(\vp\circ\pi)(0,z')+A|z|^2+3\ve\delta \log|s_D|^2_{h_D},
\end{equation}
where $A>0$ is a constant. If we shrink the $W_j$'s, still preserving the property that $\ti{V}\subset\cup_j W_j$, we can choose $A$ sufficiently large so that
$$\pi^*\alpha+\ddb\vp_j\geq 2\ve\ti{\omega},$$
holds on $W_j$ for all $j$. It will also be useful to fix slightly smaller open sets $W_j'\Subset U_j\Subset W_j$ such that $\cup_j W_j'$ still covers $\ti{V}$.
Note that since $\vp$ is smooth at the generic point of $V$, by construction all functions $\vp_j$ are also smooth in a neighborhood of the generic point of $\ti{V}\cap W_j$.

We wish to glue the functions $\vp_j$ together to produce a K\"ahler current defined in a neighborhood of $\ti{V}$ in $\ti{X}$.
This would be straightforward if the functions $\vp_j$ were continuous, thanks to a procedure of Richberg \cite{Rich},
but in our case the functions $\vp_j$ have poles along $\cup_\ell D_\ell\cap W_j$. However, as we will now show, on any nonempty intersection $W_i\cap W_j$
the differences
$\vp_i-\vp_j$ are continuous, and this is enough to apply the argument of Richberg.

So take two open sets $W_1, W_2$ in the covering, with $W_1\cap W_2\cap \ti{V}$ nonempty. Let $(z^1_1,\dots,z^1_n)$ be the coordinates on $W_1$ as above,
and $(z^2_1,\dots,z^2_n)$ the ones on $W_2$. If none of the divisors $D_\ell$ intersects $W_1\cap W_2$, then $\vp_i-\vp_j$ is clearly smooth there, so
we may assume that at least one of them, $D_\ell$ say, intersects $W_1\cap W_2$. Up to reordering the coordinates, we can write $D_\ell\cap W_1=\{z^1_\ell=0\}$ and $D_\ell\cap W_2=\{z^2_\ell=0\}$.
Therefore on $W_1\cap W_2$ we must have $z^2_\ell=u_\ell\cdot z^1_\ell$, where $u_\ell$ is a never-vanishing holomorphic function.

From the construction of the resolution $\pi$, and thanks to \eqref{ideal} and \eqref{ideal2}, we see that $\vp_1$ has weakly analytic singularities along $\cup_\ell D_\ell\cap W_1$, and the ideal defining its singularities is principal, namely on $W_1$ we have
\begin{equation}\label{expans}
\vp_1=c_1\log\left(\prod_k |z^1_{i_k}|^{2\alpha^1_{i_k}}\right)+g_1,
\end{equation}
where $c_1, \alpha^1_{i_k}$ are constants, $g_1$ is a continuous function, and $\{\prod_k |z^1_{i_k}|^{2\alpha^1_{i_k}}=0\}$ equals $\cup_\ell D_\ell\cap W_1$. A similar formula holds for $\vp_2$ on $W_2$, and since
$\vp_1|_{\ti{V}}=\vp_2|_{\ti{V}}$ and $z^j_{i_k}|_{\ti{V}}\not\equiv 0, j=1,2$, we must have $c_1=c_2$ and $\alpha^1_{i_k}=\alpha^2_{i_k}$.
Thanks to the relations $z^2_{i_k}=u_{i_k}\cdot z^1_{i_k}$ we see that $\vp_1-\vp_2$ is indeed continuous on $W_1\cap W_2$. Furthermore, on $W_1\cap W_2$ the function $\max\{\vp_1,\vp_2\}$ differs from $\vp_1$ and $\vp_2$ by a continuous function, and hence it has weakly analytic singularities of the same type as $\vp_1$ and $\vp_2$, and restricts to $\hat{\vp}$ on $\ti{V}\cap W_1\cap W_2$.

Now, for the convenience of the reader, we recall the gluing argument of Richberg \cite{Rich}.
We start by considering two open sets $W_1, W_2$ in the covering with $W_1'\cap W_2'\cap \ti{V}$ nonempty,
and fix a compact set $K\subset \ti{V}$ with $(W_1'\cup W_2')\cap \ti{V}\subset K\subset (U_1\cup U_2)\cap \ti{V}$.
Let $M_1=K\cap\de U_2, M_2=K\cap\de U_1,$ so that $M_1$ and $M_2$ are disjoint compact subsets of $\ti{V}$.  This setup is depicted in figure~\ref{fig: 3}.
 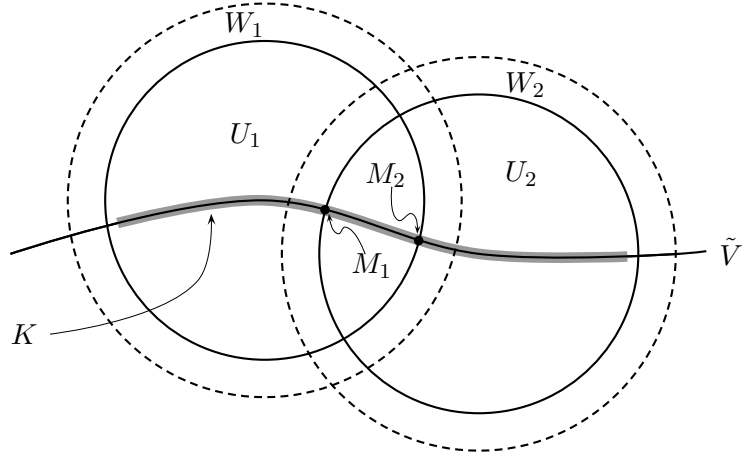
\begin{figure}[h]
\begin{center}
\psset{unit=.007in}
\begin{pspicture}(-250,-260)(250,50)
	\newgray{llgray}{.60}
	\pscurve(-250,-120)(-60, -80)(100,-120)(250,-120)(270,-118)
	\pscurve[linecolor=llgray ,linewidth=4pt](-170,-97)(-60,-80)(100,-120)(211,-122)
	\pscurve(-250,-120)(-60, -80)(100,-120)(250,-120)
	\pscurve[linewidth=.75, arrowsize=3pt]{c->}(-220,-180)(-110,-140)(-100,-90)
	\pscurve[linewidth=.75]{c->}(35,-70)(40,-95)(50,-90)(54,-105)
	\pscurve[linewidth=.75]{c->}(15,-120)(0,-100)(-5,-105)(-12,-90)
	\psellipse(-60,-80)(120,120)
	\psellipse(100,-120)(120,120)
	\psset{linestyle=dashed,dash=3pt 2pt}
	\psellipse(-60,-80)(148,148)
	\psellipse(100,-120)(148,148)
	\uput[100](-70,-50){$U_{1}$}
	\uput[0](110,-60){$U_{2}$}
	\uput[0](270,-118){$\ti{V}$}
	\uput[4](-100,50){$W_{1}$}
	\uput[5](110,5){$W_{2}$}
	\uput[180](-220,-180){$K$}
	\uput[120](45,-80){$M_{2}$}
	\uput[270](20,-110){$M_{1}$}
	\psdots(-15,-87)(55,-110)
\end{pspicture}
\caption{The setup for the local Richberg argument.   }
\label{fig: 3}
\end{center}
\end{figure}

Pick $\theta_1$ a smooth nonnegative cutoff function which is identically $1$ in a neighborhood of $M_2$ in $\ti{X}$ and
$\theta_2$ a smooth nonnegative cutoff function which is identically $1$ in a neighborhood of $M_1$ in $\ti{X}$ so that
the supports of $\theta_1$ and $\theta_2$ are disjoint.
Then, if we choose $\eta>0$ small, the functions
\begin{equation}\label{tilda}
\ti{\vp}_j=\vp_j-\eta\theta_j,
\end{equation}
$j=1,2$ have analytic singularities and satisfy $\pi^*\alpha+\ddb \ti{\vp}_j\geq \ve\ti{\omega}$ on $W_j$.
On $W_1\cap W_2$ we then define
$$\ti{\vp}_0=\max\{\ti{\vp}_1,\ti{\vp}_2\},$$
which has weakly analytic singularities, it satisfies $\pi^*\alpha+\ddb \ti{\vp}_0\geq \ve\ti{\omega}$
and it restricts to $\hat{\vp}$ on $\ti{V}\cap W_1\cap W_2$.
Consider now a neighborhood of $M_2$ in $\ti{X}$, small enough so that $\theta_1=1$ on $M_2$. Since $\vp_1,\vp_2$ agree on $\ti{V}$ and
$\vp_1-\vp_2$ is continuous on this neighborhood, we see that there exists a possibly smaller such neighborhood where
\begin{equation*}
\tilde{\phi}_{1} = \phi_{1} - \eta < \phi_{2} = \tilde{\phi}_{2},
\end{equation*}
so that $\ti{\vp}_0=\ti{\vp}_2$ there. Similarly, on any sufficiently small neighborhood of $M_1$ we have $\ti{\vp}_0=\ti{\vp}_1$.
Therefore there is an open neighborhood $W_0$ of $K\cap \ov{U_1}\cap \ov{U_2}$ in $\ti{X}$ such that
$\ti{\vp}_0=\ti{\vp}_1$ on $W_0\backslash \ov{U_2}$ and $\ti{\vp}_0=\ti{\vp}_2$ on $W_0\backslash \ov{U_1}$.
Therefore we can define
$$W'=W_0\cup (U_1\backslash \ov{U_2})\cup (U_2\backslash \ov{U_1}),$$
which is a neighborhood of $K$ in $\ti{X}$,
and define a function
$\vp'$ on $W'$ to be equal to $\ti{\vp}_0$ on $W_0$, equal to $\ti{\vp}_1$ on $U_1\backslash \ov{U_2}$ and equal to $\ti{\vp}_2$ on $U_2\backslash \ov{U_1}$.
Then $\vp'$ satisfies $\pi^*\alpha+\ddb\vp'\geq \ve\ti{\omega}$, it has weakly analytic singularities of the same type as $\vp_1$ and $\vp_2$,
and restricts to $\hat{\vp}$ on $W'\cap \ti{V}$. Clearly, $W'$ contains
$(W_1'\cup W_2')\cap \ti{V}$.

We then replace $W_1$ and $W_2$ with $W'$, replace $\vp_1$ and $\vp_2$ with $\vp'$, and repeat the same procedure with two other open sets in this new covering.
After at most $N$ such steps, we end up with an open neighborhood $W$ of $\ti{V}$ in $\ti{X}$ with a function $\vp''$ defined on $W$ which satisfies
$\pi^*\alpha+\ddb\vp''\geq\ve'\ti{\omega}$ for some $\ve'>0$, it has weakly analytic singularities,
and it restricts to $\hat{\vp}$ on $\ti{V}$. This ends the gluing procedure of Richberg.

Now we have a K\"ahler current $\pi^*\alpha+\ddb\vp''$ defined on $W$, with $\phi''=-\infty$ on all the divisors $D_\ell$ which intersect $W$,
and the generic Lelong number of $\phi''$ along each such divisor is strictly positive. On the other hand we also have the function $F\circ\pi$ on $\ti{X}$, singular along $\cup_\ell D_\ell\cup\cup_i \ti{Y}_i\cup\ti{V}$, where $\ti{Y}_i$ is the proper transform of $Y_i$ in $\ti{X}$, and
 which
satisfies $\pi^*\alpha+\ddb(F\circ\pi)\geq \gamma \pi^*\omega$ for some small $\gamma>0$. Hence
$$\pi^*\alpha+\ddb(F\circ\pi+\gamma\delta\log|s_D|^2_{h_D})\geq \gamma \ti{\omega},$$
is a K\"ahler current. Let $\ti{F}=F\circ\pi+\gamma\delta\log|s_D|^2_{h_D}$.
Since the class $[\pi^*\alpha]$ is nef, for any small $\nu>0$ there exists a smooth function $\rho_\nu$ on $\ti{X}$ such that
$$\alpha+\ddb\rho_\nu\geq -\nu\gamma\ti{\omega},$$
holds on $\ti{X}$, and we can normalize $\rho_\nu$ so that it is strictly positive.
Then we have
 \begin{equation*}
 \pi^*\alpha + \ddb \left(\nu \ti{F} + (1-\nu)\rho_{\nu}\right) \geq \nu^{2}\gamma \omega.
 \end{equation*}
Let us shrink $W$ slightly, so that $\vp''$ is defined on an open set containing $W$ in its interior. Thanks to \cite[Lemma 2.1]{DP}, we can find an $A_0\ti{\omega}$-PSH function $\psi$ on $\ti{X}$, for some $A_0>0$, with analytic singularities along $\cup_i\ti{Y}_i$. We then let $\phi'''=\phi''+\frac{\ve'}{2A_0}\psi$, so that $\pi^*\alpha+\ddb\phi'''\geq \frac{\ve'}{2}\ti{\omega}$ on $W$ and $\phi'''$ has weakly analytic singularities along $(\cup_\ell D_\ell\cup\cup_i \ti{Y}_i)\cap W$.

Pick any point $x\in (\cup_\ell D_\ell\cup \cup_i\ti{Y}_i)\cap\de W$, and fix a coordinate neighborhood $U$ of $x$ where $\vp'''$ is defined, with $U\cap\ti{V}=\emptyset$.
We can then take a principalization $\hat{\pi}:\hat{U}\to U$ of the ideal sheaves of the singularities of $\phi'''$ and $\ti{F}$ on $U$, obtained by a composition of smooth blowups, so that
the pullbacks of these ideal sheaves to $\hat{U}$ become principal with simple normal crossings support.
Therefore $\hat{U}$ is covered by local coordinates as before, where we can write
$$\hat{\pi}^*\vp'''=c_1\log\left(\prod_k |z_{i_k}|^{2\alpha_{i_k}}\right)+g_1,$$
$$\hat{\pi}^*\ti{F}=c_2\log\left(\prod_k |z_{i_k}|^{2\beta_{i_k}}\right)+g_2,$$
where $c_1,c_2, \alpha_{i_k},\beta_{i_k}$ are constants, $g_1,g_2$ are continuous functions.
Note that crucially the ideal sheaves which define the singularities of $\phi'''$ and $\ti{F}$ have the same support on $U$, which implies that $\alpha_{i_k}>0$ iff $\beta_{i_k}>0$. Therefore, if $\nu>0$ is sufficiently small, we will have that $\nu \ti{F}+(1-\nu)\rho_{\nu}>\phi'''$ on a possibly smaller open neighborhood of $x$.

Repeating this argument on a finite covering of $(\cup_\ell D_\ell\cup \cup_i\ti{Y}_i)\cap \de W$, we can find $\nu>0$ and a neighborhood
$Z$ of $\cup_\ell D_\ell\cup \cup_i\ti{Y}_i$ such that $\nu \ti{F}+(1-\nu)\rho_{\nu}>\phi'''$ on $Z\cap \de W$.
Since $\ti{F}>-\infty$ on $(\de W)\backslash Z$, we can choose $A> 0$ large enough such that on
$(\de W)\backslash Z$ we have $\nu \ti{F}+(1-\nu)\rho_{\nu}>\phi'''-A$. Altogether, this means that $\nu \ti{F}+(1-\nu)\rho_{\nu}>\phi'''-A$
holds in a whole neighborhood of $\de W$.

Therefore we can finally define
\begin{displaymath}
   \ti{\Phi} = \left\{
     \begin{array}{ll}
     \max\{ \phi''',  \nu \ti{F} + (1-\nu)\rho_{\nu}+A\} &  \text{ on }W\\
       \nu F + (1-\nu)\rho_{\nu}+A, &  \text{ on } \ti{X}\backslash W,
     \end{array}
   \right.
\end{displaymath}
which is defined on the whole of $\ti{X}$ and satisfies $\ti{T}_1=\pi^*\alpha+\ddb\ti{\Phi}\geq \ve''\ti{\omega}$ for some $\ve''>0$. Furthermore, since $\ti{F}=-\infty$ on $\ti{V}$, while $\vp'''$ is continuous near the generic point of $\ti{V}$, we see that $\ti{\Phi}=\vp'''$ in a neighborhood of the generic point of $\ti{V}$. Therefore, $\ti{\Phi}|_{\ti{V}}$ is smooth in a neighborhood (in $\ti{V}$) of the generic point of $\ti{V}$.

The pushforward $\ti{T}=\pi_*\ti{T}_1$ is then a K\"ahler current on $X$ in the class $[\alpha]$ whose restriction to $V$ is smooth in a neighborhood (in $V$) of the generic point
of $V$.
This completes the proof of Theorem \ref{extend}.
 \end{proof}

\section{Finite time singularities of the K\"ahler-Ricci flow}\label{sectkrf}
In this section we study finite time singularities of the K\"ahler-Ricci flow, and give the proof of Theorem \ref{finitet}.

Let $(X, \omega_{0})$ be a compact K\"ahler manifold of dimension $n$, and consider K\"ahler metrics $\omega(t)$ evolving under the K\"ahler-Ricci flow \eqref{KRF} with initial metric $\omega_{0}$.
Suppose that the maximal existence time of the flow is $T<\infty$. Then the class $[\alpha] = [\omega_{0}]-Tc_{1}(X)$ is nef, but not K\"ahler.  Moreover, assume that $\int_X\alpha^n>0$, so the volume of $(X,\omega(t))$ does not go to zero as $t\to T^-$.  Let $\alpha$ be a smooth representative of this limiting class, and set
\begin{equation*}
\begin{aligned}
&\hat{\omega}_{t} = \frac{1}{T}\left((T-t)\omega_{0} + t\alpha\right) \in [\omega_{0}]-tc_{1}(X)\\
& \chi = \ddt \hat{\omega}_{t} = \frac{1}{T}( \alpha - \omega_{0}) \in -c_{1}(X).
\end{aligned}
\end{equation*}
Let $\Omega$ be a volume form so that $\ddb \log \Omega = \chi$, and $\int_{X}\Omega = \int_{X}\omega_{0}^{n}$.  Then the K\"ahler-Ricci flow can be written as the parabolic complex Monge-Amp\`ere equation
\begin{equation*}
\ddt \vp = \log \frac{ \left( \hat{\omega}_{t}+ \ddb \vp \right)^{n}}{\Omega}, \qquad \hat{\omega}_{t}+ \ddb \vp>0, \qquad \vp(0)=0,
\end{equation*}
so that
\begin{equation*}
\omega(t) = \hat{\omega}_{t}+ \ddb \vp,
\end{equation*}
solves \eqref{KRF}.
Since the limiting class is nef and has positive self-intersection, Theorem \ref{usefull} shows that there exists a K\"ahler current
\begin{equation}\label{eq: epsilon def}
R = \alpha + \ddb \psi \geq \ve \omega_0,
\end{equation}
which has analytic singularities, and such that
\begin{equation*}
\{ \psi = -\infty\} = E_{+}(R) = E_{nK}(\alpha).
\end{equation*}
By subtracting a constant to $\psi$, we can assume that $\sup_X\psi\leq 0$. Recall from \cite{Bou} that an $\alpha$-PSH function $\vp$ has said to have minimal singularities if for every $\alpha$-PSH function $\eta$ there is a constant $C$ such that $\vp\geq \eta-C$ on $X$. Then we have

\begin{thm}\label{main2}
There is a closed positive real $(1,1)$ current $\omega_T$ on $X$ in the class $[\alpha]$, which is smooth precisely away from $E_{nK}(\alpha)$ and has minimal singularities, such that as $t\to T^-$
we have that $\omega(t)$ converges to $\omega_T$ in $C^\infty_{\mathrm{loc}}(X\backslash E_{nK}(\alpha))$ as well as currents on $X$.
\end{thm}

Theorem \ref{finitet} follows easily from this result.
\begin{proof}[Proof of Theorem \ref{finitet}]
Theorem \ref{main2} says that no singularities develop on $X\backslash E_{nK}(\alpha)$. Thanks to the main Theorem \ref{main}, $E_{nK}(\alpha)=\Null(\alpha)$. Finally, we show that the metrics $\omega(t)$ must develop singularities everywhere along $\Null(\alpha)$. Indeed, assume that there is a point $x\in\Null(\alpha)$ with an open neighborhood $U$ where the metrics $\omega(t)$ converge smoothly to a limit K\"ahler metric $\omega_U$ on $U$. Since $x\in\Null(\alpha)$, there is an irreducible $k$-dimensional analytic subvariety $V\subset X$ with $x\in V$ and $\int_V \alpha^k=0$. Then
$$\int_V\omega(t)^k\geq \int_{V\cap U}\omega(t)^k\to \int_{V\cap U}\omega_U^k>0,$$
but at the same time $\int_V\omega(t)^k\to \int_V\alpha^k=0$, a contradiction.
\end{proof}
The same proof shows that if $T<\infty$ is a finite time singularity with $\int_X\alpha^n=0$, then singularities develop on the whole of $X$.

As a side remark, we have the following:
\begin{prop} Let $(X^n,\omega_0)$ be a compact K\"ahler manifold with nonnegative Kodaira dimension. Then if the K\"ahler-Ricci flow \eqref{KRF} develops a singularity at a finite time $T$, we must necessarily have that
$$\int_X(\omega_0-Tc_1(X))^n>0,$$  and Theorem \ref{finitet} then applies.
\end{prop}
\begin{proof}
Indeed, nonnegative Kodaira dimension means that $H^0(X,\ell K_X)\neq 0$ for some $\ell\geq 1$. Since $c_1(K_X)=-c_1(X)$, this
implies that the class $-Tc_1(X)$ is pseudoeffective (i.e. it contains a closed positive current $S$),
and therefore the class $[\alpha]=[\omega_0]-Tc_1(X)$ is big since it contains the K\"ahler current $S+\omega_0$. Then \cite[Theorem 4.7]{Bo} shows that the volume $v(\alpha)$ is strictly positive, and since $[\alpha]$ is also nef we can apply  \cite[Theorem 4.1]{Bo} and conclude that $\int_X\alpha^n=v(\alpha)>0$.
\end{proof}
If $X$ is projective then it is enough to assume that $X$ is not uniruled, using the same argument as above together with the main theorem of \cite{BDPP}.

Let us also remark that when $n=2$ if $[\alpha]=\omega_0-Tc_1(X)$ is the limiting class as above, then $\Null(\alpha)$ is a disjoint union of finitely many $(-1)$-curves (see \cite[Theorem 8.3]{SW}) which can be contracted to give a new K\"ahler surface $Y$ and $[\alpha]$ is the pullback of a K\"ahler class on $Y$. The K\"ahler-Ricci flow can then be restarted on $Y$, and the whole process is continuous in the Gromov-Hausdorff topology thanks to work of Song-Weinkove \cite{SW2}.

Now we start with the proof of Theorem \ref{main2}. We first need several lemmas.
\begin{lem}\label{lem1}
There is a constant $C>0$ such that on $X\times [0,T)$ we have
\begin{enumerate}
\item[(i)]  $\displaystyle{\dot \varphi(t)\leq C}$,
\item[(ii)] $\displaystyle{ \omega^n\leq C\Omega}$,
\item[(iii)]  $\displaystyle{\varphi(t)\leq C}$.
\end{enumerate}
\end{lem}
The proof is extremely easy (see \cite[Lemma 7.1]{SW}). Note that (i) is equivalent to (ii), and that integrating (i) from $0$ to $T$ gives (iii).

\begin{lem}\label{lem2}
There is a constant $C>0$ such that on $X\times [0,T)$ we have
\begin{enumerate}
\item[(i)]  $\displaystyle{\dot \varphi(t)\geq C\psi -C}$,
\item[(ii)] $\displaystyle{ \omega^n\geq C^{-1}e^{C\psi}\Omega}$.
\end{enumerate}
\end{lem}
\begin{proof}
It is enough to prove (i), since (ii) is clearly equivalent to (i).

To prove (i), note that
$$\hat{\omega}_t+\ddb\psi=\frac{1}{T}\left((T-t)(\omega_0+\ddb\psi)+t(\alpha+\ddb \psi) \right).$$
Now we have
$$\omega_0+\ddb\psi\geq \omega_0-\alpha\geq -C\omega_0,\quad \alpha+\ddb \psi=R\geq\ve\omega_0,$$
where these inequalities are in the sense of currents.
Therefore, there is a small $\delta>0$ such that for $t\in [T-\delta,T)$ we have
\begin{equation}\label{key}
\hat{\omega}_t+\ddb\psi \geq \frac{\ve}{2}\omega_0,
\end{equation}
again as currents.
Let $$Q=\dot{\vp}+A\vp-A\psi+Bt,$$
where $A$ and $B$ are large constants to be fixed soon. Our goal is to show that $Q\geq -C$ on $X\times [0,T)$, since then
Lemma \ref{lem1} (iii) implies (i).

We fix the constant $A$ large enough so that
$$(A-1)(\hat{\omega}_t+\ddb\psi)+\chi\geq \omega_0,$$
holds on $[T-\delta,T)$. To fix the value of $B$, we first use the arithmetic-geometric mean inequality to get on $[T-\delta, T)$
\begin{equation}\label{amgm}
\frac{\ve}{2}\tr{\omega}{\omega_0}\geq C_0^{-1}\left(\frac{\Omega}{\omega^n}\right)^{1/n},
\end{equation}
where $C_0$ is a uniform constant.
But since the function $y\mapsto A\log y-C_0^{-1}y^{1/n}$ is bounded above for $y>0$, we have that
\begin{equation}\label{logbd}
C_0^{-1}\left(\frac{\Omega}{\omega^n}\right)^{1/n}+A\log\frac{\omega^n}{\Omega}\geq -C_1,
\end{equation}
for another uniform constant $C_1$. We then set $B=C_1+An$.

From the definition of $Q$ we see that $Q\geq -C$ holds on $X\times [0,T-\delta]$, for a uniform constant $C$. Therefore, to prove that $Q\geq -C$ on $X\times [0,T)$, it suffices to show that given any $T-\delta<T'<T$, the minimum of $Q$ on $X\times [0,T']$ is always achieved
on $[0,T-\delta]$. Let then $(x,t)$ be a point in $X\times (T-\delta,T']$ where $Q$ achieves a minimum on $X\times[0,T']$.
Since $Q$ approaches $+\infty$ along $E_{nK}(\alpha)$, we conclude that $x\notin E_{nK}(\alpha)$. At $(x,t)$ we have
\[\begin{split}
0\geq \left(\frac{\de}{\de t}-\Delta\right)Q&=\tr{\omega}{\chi}+A\dot{\vp}-An+A\tr{\omega}{(\hat{\omega}_t+\ddb\psi)}+B\\
&=\tr{\omega}{(\hat{\omega}_t+\ddb\psi)}+\tr{\omega}{\left((A-1)(\hat{\omega}_t+\ddb\psi)+\chi\right)}\\
&\ \ \ \ +A\log\frac{\omega^n}{\Omega}-An+B\\
&\geq \frac{\ve}{2}\tr{\omega}{\omega_0}+A\log\frac{\omega^n}{\Omega}+\tr{\omega}{\omega_0}-An+B,
\end{split}\]
where $\Delta=\Delta_{\omega(t)}$ is the Laplacian of the moving metric.
Thanks to \eqref{amgm}, \eqref{logbd}, we conclude that
\[\begin{split}
0\geq \left(\frac{\de}{\de t}-\Delta\right)Q&\geq \tr{\omega}{\omega_0}-C_1-An+B=\tr{\omega}{\omega_0}>0,
\end{split}\]
a contradiction.
\end{proof}

Integrating Lemma \ref{lem2} (i) from $0$ to $T$ gives the bound $\varphi(t)\geq C\psi -C.$
We now sharpen this:

\begin{lem}\label{lem3}
There is a constant $C>0$ such that on $X\times [0,T)$ we have
\begin{equation}\label{c0}
\varphi(t)\geq \psi -C.
\end{equation}
\end{lem}
\begin{proof}
Let $Q=\vp-\psi+At$, for $A$ a large constant to be determined.
We need to show that $Q\geq -C$ on $X\times[0,T)$. Clearly,  $Q\geq -C$  holds on $X\times [0,T-\delta]$ (where $\delta$ is as in Lemma \ref{lem2}), so we
fix $T-\delta<T'<T$ and suppose that $Q$ achieves its minumum on $X\times [0,T']$ at $(x,t)$ with $t\in (T-\delta,T']$.
Since $\psi$ approaches $-\infty$  along $E_{nK}(\alpha)$, we have that $x\in X\backslash E_{nK}(\alpha)$.
At $(x,t)$ we then have
\[\begin{split}
0\geq\frac{\de Q}{\de t}&=\log\frac{(\hat{\omega}_t+\ddb\psi+\ddb Q)^n}{\Omega}+A\\
&\geq \log\frac{(\hat{\omega}_t+\ddb\psi)^n}{\Omega}+A\\
&\geq \log \frac{(\ve\omega_0/2)^n}{\Omega}+A\geq -C+A,
\end{split}\]
using \eqref{key}. If we choose $A>C$, this gives a contradiction.
\end{proof}

\begin{lem}\label{c2}
There is a constant $C>0$ such that on $X\times [0,T)$ we have
\begin{equation}\label{lapl}
\tr{\omega_0}{\omega}\leq Ce^{-C\psi}.
\end{equation}
\end{lem}
\begin{proof}
Let
$$Q=\log\tr{\omega_0}{\omega}-A\vp+A\psi,$$
where $A$ is a large constant to be determined soon.
Thanks to Lemma \ref{lem1} (iii), we will be done if we show that $Q\leq C$ on $X\times[0,T)$.
Again, from the definition of $Q$ we see that $Q\leq C$ holds on $X\times [0,T-\delta]$, for a uniform constant $C$.
Fix then any $T-\delta<T'<T$ and suppose that $Q$ achieves its maximum on $X\times[0,T']$ at $(x,t)$ with $t\in [T-\delta, T']$. Then at $(x,t)$ we can
use \eqref{key} to estimate
\[\begin{split}
0\leq \left(\frac{\de}{\de t}-\Delta\right)Q&\leq C\tr{\omega}{\omega_0}+A\log\frac{\Omega}{\omega^n}+An-A\tr{\omega}{(\hat{\omega}_t}+\ddb\psi)\\
&\leq (C-A\ve/2)\tr{\omega}{\omega_0}+A\log\frac{\Omega}{\omega^n}+An,
\end{split}\]
where we used the standard parabolic $C^2$ ``Aubin-Yau'' calculation (see e.g. \cite[Proposition 2.5]{SW}).
Hence if we choose $A$ large so that $C-A\ve/2\leq -1$, we see that at $(x,t)$ we have
$$\tr{\omega}{\omega_0}\leq C\log\frac{\Omega}{\omega^n}+C.$$
Hence at the maximum of $Q$,
$$\tr{\omega_0}{\omega} \le \frac{1}{(n-1)!} (\tr{\omega}{\omega_0})^{n-1} \frac{\omega^n}{\omega_0^n} \le
C\frac{\omega^n}{\Omega} \left(\log \frac{\Omega}{\omega^n}\right)^{n-1}+C\leq C,$$
because we know that $\frac{\omega^n}{\Omega}\leq C$ (Lemma \ref{lem1} (ii)) and $x\mapsto x|\log x|^{n-1}$
is bounded above for $x$ close to zero. Thanks to Lemma \ref{lem3}, this implies that $Q$ is bounded from above at its maximum, and we are done.
\end{proof}

Related estimates can be found in \cite{BG} under slightly different assumptions.

\begin{proof}[Proof of Theorem \ref{main2}]
Fix any compact set $K\subset X\backslash E_{nK}(\alpha)$. Thanks to Lemmas \ref{lem2} (ii) and \ref{c2}, there is a constant $C_K>0$ such that
$$C_K^{-1} \omega_0 \le \omega(t) \le C_K\omega_0 \quad \textrm{on } K \times [0,T).$$
Applying the local higher order estimates of Sherman-Weinkove \cite{ShW}, we get uniform $C^{\infty}$ estimates for $\omega(t)$ on compact subsets of $X\backslash E_{nK}(\alpha)$.  This implies that
given every compact set $K$ there exists a constant $C_K$ such that
$$\ddt{} \omega = - \Ric(\omega) \le C_K \omega,$$
which implies that $e^{-C_Kt} \omega(t)$ is decreasing in $t$ as well as being bounded from below.  This implies that $\omega(t)$ has a limit as $t \rightarrow T$, and since we have uniform estimates away from $E_{nK}(\alpha)$, we see that $\omega(t)$ converges in $C^{\infty}$ on compact sets to a smooth K\"ahler metric $\omega_{T}$ on $X\backslash E_{nK}(\alpha)$.
Furthermore, by weak compactness of currents, it follows easily that $\omega_T$ extends as a closed positive current on $X$ and $\omega(t)$ converges to $\omega_T$ as currents on $X$.

Finally we show that the limiting current $\omega_T$ on $X$ has minimal singularities (which implies for example that it has vanishing Lelong numbers \cite{Bou}). This argument is an adaptation of a similar one in \cite{Z1}. If we let $\eta$ be any $\alpha$-PSH function, we need to show that there is a constant $C$ (which depends on $\eta$), such that
\begin{equation}\label{goalll}
\vp\geq \eta -C,
\end{equation}
holds on $X\times [0,T)$.
Using Demailly regularization, we obtain functions $\eta_\ve$, with analytic singularities, with $\alpha+\ddb\eta_\ve\geq -\ve\omega_0$, which decrease to $\eta$. Let
$$Q=\left((\vp+(T-t)\dot{\vp}+nt)+\ve(\vp-t\dot{\vp}+nt)-\eta_\ve\right),$$
and, wherever $Q$ is smooth, calculate
\[\begin{split}
\left(\frac{\de}{\de t}-\Delta\right)Q=\tr{\omega}{(\alpha+\ve\omega_0+\ddb\eta_\ve)}\geq 0.
\end{split}\]
Note that for each time $t$ the quantity $Q$ is bounded below (and smooth at a point where its minimum is achieved), hence the minimum principle (together with the fact that $\eta_\ve\leq C$ independent of $\ve$) implies that $Q\geq -C$, where the constant $C$ does not depend on $\ve$,
or in other words
$$(1+\ve)\vp+(T-t-\ve t)\dot{\vp}\geq \eta_\ve-C\geq \eta-C.$$
Letting $\ve\to 0$ and using that $\dot{\vp}\leq C$, we obtain \eqref{goalll}.
\end{proof}

We now remark that the set of points where the K\"ahler-Ricci flow develops singularities can also be characterized using the curvature of the evolving metrics. We are grateful to Zhou Zhang for pointing out this result to us.

Following Enders-M\"uller-Topping \cite{EMT}, we will call $\Sigma$ the complement of the set of points $x\in X$ such that there exists an open neighborhood $U$ of $x$ and a constant $C>0$ with
$|\mathrm{Rm}(t)|_{g(t)}\leq C$ on $U\times [0,T)$.
Also, call $\Sigma'$ the complement of the set of points $x\in X$ such that there exists an open neighborhood $U$ of $x$ and a constant $C>0$ with
$R(t)\leq C$ on $U\times [0,T)$, where $R(t)$ is the scalar curvature of $\omega(t)$.

\begin{thm}\label{singul}
Let the setup be the same as in Theorem \ref{finitet}.
Then
$$\Null(\alpha)=\Sigma=\Sigma'.$$
\end{thm}
\begin{proof}
The inclusion $\Sigma'\subset\Sigma$ is trivial, and thanks to Theorem \ref{main2} we have the inclusion
$\Sigma\subset\Null(\alpha).$
It then suffices to prove that if $R(t)\leq C$ on $U\times [0,T)$, then on a smaller open neighborhood $U'$ of $x$ we have that
$\omega(t)$ converge smoothly to a K\"ahler metric $\omega_{U'}$ as $t\to T^-$.

The proof is contained in the work of Zhang \cite{Z1, Z3} (see also the exposition \cite[Theorem 7.6]{SW}).
In general, $R\geq -C$ along the flow.
From the flow equation
$$\dot{\vp}=\log\frac{\omega^n}{\Omega},\quad \frac{\de}{\de t}\dot{\vp}=-R,$$
so $\frac{\de}{\de t}\dot{\vp}\leq C$ on $M\times[0,T)$. Integrating from $0$ to $T$ we get $\dot{\vp}\leq C$ and $\vp\leq C$ on $M\times[0,T)$.
Furthermore, on $U\times[0,T)$ we have $\left|\frac{\de}{\de t}\dot{\vp}\right|\leq C$ and hence $|\dot{\vp}|+|\vp|\leq C$. The quantity $H=t\dot{\vp}-\vp-nt$ is therefore bounded on $U\times[0,T),$ and satisfies
$$\left(\frac{\de}{\de t}-\Delta\right)H=-\tr{\omega}{\omega_0}.$$
On the other hand a ``Schwarz-Lemma''-type calculation (see for example \cite{ST, SW}) gives
$$\left(\frac{\de}{\de t}-\Delta\right)\log\tr{\omega}{\omega_0}\leq C\tr{\omega}{\omega_0}.$$
Therefore if we let $Q=\log\tr{\omega}{\omega_0}+AH$, then
$$\left(\frac{\de}{\de t}-\Delta\right)Q\leq-\tr{\omega}{\omega_0}\leq 0,$$
provided $A$ is large enough. Therefore given any $T'<T$, the maximum of $Q$ on $M\times[0,T']$ is achieved at $t=0$, and hence
$$\log\tr{\omega}{\omega_0}+A(t\dot{\vp}-\vp-nt)\leq C.$$
This means that on $M\times[0,T)$ we have
$$\tr{\omega}{\omega_0}\leq Ce^{-At\dot{\vp}+A\vp+Ant}\leq Ce^{-At\dot{\vp}}.$$
But $\omega^n/\Omega\leq C$ thanks to Lemma \ref{lem1} (ii), and so
$$\tr{\omega_0}{\omega}\leq C(\tr{\omega}{\omega_0})^{n-1} \frac{\omega^n}{\Omega}\leq Ce^{-(n-1)At\dot{\vp}}.$$
On $U\times[0,T)$ we therefore have $C^{-1}\omega_0\leq \omega\leq C\omega_0$. The local higher order estimates of Sherman-Weinkove \cite{ShW},
then give us $C^\infty$ estimates for $\omega$ on $U'\times[0,T)$ for a slightly smaller open neighborhood $U'$ of $x$, and from here we conclude as in the proof of Theorem \ref{main2}.
\end{proof}

\section{Degenerations of Ricci-flat K\"ahler metrics}\label{sectcy}
In this section we study degenerations of Ricci-flat K\"ahler metrics on Calabi-Yau manifolds, and prove Theorem \ref{main1}.

Let $X$ be a Calabi-Yau manifold, which is a compact K\"ahler manifold with $c_1(X)=0$ in $H^2(X,\mathbb{R})$.
Let $\mathcal{K}\subset H^{1,1}(X,\mathbb{R})$ be the K\"ahler cone of $X$. Then Yau's theorem \cite{Y} says that for every K\"ahler class $[\alpha]\in\mathcal{K}$
there exists a unique K\"ahler metric $\omega\in[\alpha]$ with $\Ric(\omega)=0$.
Let $\alpha$ be a smooth closed real $(1,1)$ form on $X$ with $[\alpha]\in \de\mathcal{K}$ nef and $\int_X\alpha^n>0$.
Let $[\alpha_t]:[0,1]\to\overline{\mathcal{K}}$ be a continuous path of $(1,1)$ classes with $[\alpha_t]\in \mathcal{K}$ for $t>0$ and with $[\alpha_0]=[\alpha]$.
For $0<t\leq 1$ write $\omega_t$ for the unique Ricci-flat K\"ahler metric in the class $[\alpha_t]$.

\begin{proof}[Proof of Theorem \ref{main1}] Throughout this proof we fix $\omega$ a Ricci-flat K\"ahler metric on $X$.
To start, let us see that we can choose a continuously varying path of reference $(1,1)$ forms $\alpha_t$ cohomologous to $[\alpha_t]$ with $\alpha_0=\alpha$.
Indeed, it is easy to find $(1,1)$ forms $\ti{\alpha}_t$ representing the classes $[\alpha_t]$ and varying continuously in $t$.
Since $\ti{\alpha}_0$ and $\alpha$ are cohomologous, we may write $\ti{\alpha}_0=\alpha+\ddb \eta$, and then let $\alpha_t=\ti{\alpha}_t-\ddb\eta$.
In general the forms $\alpha_t$ do not satisfy any positivity property.
Recall the following construction from \cite{BEGZ}: for any $x\in X$ let
$$V_{\alpha}(x)=\sup\{\vp(x)\ |\ \vp \textrm{ is }\alpha\textrm{-PSH}, \sup_X\vp\leq 0\}$$
be the extremal function associated to the form $\alpha$. Then $\alpha+\ddb V_\alpha\geq 0$ and $V_\alpha$ has minimal singularities among all $\alpha$-PSH functions.
Similarly, for $0\leq t\leq 1$ we let $V_t=V_{\alpha_t}$, so that $V_0=V_\alpha$.
Then the function $V_\alpha$ is continuous on $X\backslash E_{nK}(\alpha)$, while if $t>0$ then $V_t$ is continuous everywhere on $X$.

Using the theorem of Demailly-P\u{a}un together with Demailly's regularization, we see that there exists an $\alpha$-PSH function $\psi$ with analytic singularities, with $\sup_X\psi=0$ and with
$\alpha+\ddb\psi\geq \delta \omega$ on $X$, for some $\delta>0$. Furthermore, thanks to Theorem \ref{usefull}, the function $\psi$ can be chosen to be smooth on $X\backslash E_{nK}(\alpha)$
(and singular on $E_{nK}(\alpha)$). In general $V_\alpha$ has strictly weaker singularities than $\psi$, i.e. it need not be singular on $E_{nK}(\alpha)$. For example
if $Y$ is a compact K\"ahler surface, $\pi:X\to Y$ is the blowup of a point, $E\subset X$ is the exceptional divisor, and $\alpha=\pi^*\omega$ for a K\"ahler metric $\omega$ on $Y$,
then $V_\alpha$ vanishes identically while $\psi$ is singular along $E$.

Note that for $t$ sufficiently small, we will have
\begin{equation}\label{est4}
\alpha_t+\ddb\psi = \alpha+\ddb\psi+(\alpha_t-\alpha)\geq \delta\omega/2.
\end{equation}
Therefore, from the definition of $V_t$, we have that
\begin{equation}\label{est3}
V_t\geq \psi,
\end{equation}
holds on $X$ for all $t$ sufficiently small.

For $t>0$ we can write $\omega_t=\alpha_t+\ddb\vp_t$, with $\sup_X\vp_t=0$. Then the functions $\vp_t$ solve the complex Monge-Amp\`ere equation
\begin{equation}\label{ma}
\omega_t^n=(\alpha_t+\ddb\vp_t)^n=c_t \omega^n,
\end{equation}
where the constant $c_t$ equals $\int_X \alpha_t^n/\int_X\omega^n$ (and is therefore bounded uniformly away from zero and infinity). The crucial result is then the following estimate of Boucksom-Eyssidieux-Guedj-Zeriahi
\cite[Theorem 4.1, Remark 4.5]{BEGZ}
\begin{equation}\label{est}
\vp_t\geq V_t-C,
\end{equation}
for a uniform constant $C$ independent of $t$. This is generalization of a seminal result of Ko\l odziej \cite{Ko}.
Thanks to \eqref{est3}, we conclude that
\begin{equation}\label{est2}
\vp_t\geq \psi-C,
\end{equation}
for all $t$ sufficiently small.
Furthermore \cite[Theorem 4.1]{BEGZ} also
gives us a unique $\alpha$-PSH function $\vp_0$ with $\sup_X\vp_0=0$ solving
$$\langle\alpha+\ddb\vp_0\rangle^n=c_0 \omega^n,$$
on $X\backslash E_{nK}(\alpha)$, and with $\vp_0\geq V_\alpha-C$ (here $\langle \cdot\rangle^n$ is the non-pluripolar product defined in \cite{BEGZ}). Now the same simple argument as \cite[Lemma 5.3]{BEGZ} shows that as $t\to 0$ the functions $\vp_t$ converge to $\vp_0$ in $L^1(X)$. To get higher order estimates of $\vp_t$ on $X\backslash E_{nK}(\alpha)$ we proceed as follows.
Let $\Delta_t$ be the Laplacian of $\omega_t$, $A$ a large constant, and $Q_t=\log \tr{\omega}{\omega_t} -A\vp_t+A\psi$. A standard calculation (see e.g. \cite{Y}) gives
$$\Delta_tQ_t\geq -C\tr{\omega_t}{\omega} -C -An+A\tr{\omega_t}{(\alpha_t+\ddb\psi)}.$$
We can assume that $t$ is sufficiently small so that \eqref{est4} holds, and choose $A$ sufficiently large so that
$$\Delta_tQ_t\geq \tr{\omega_t}{\omega}-C.$$
Note that for any $t>0$ the function $Q_t$ approaches $-\infty$ along $E_{nK}(\alpha)$, so its maximum is achieved
on its complement. For any sufficiently small $t$, let $x\in X\backslash E_{nK}(\alpha)$ be a point where the maximum of $Q_t$ is achieved.
Then $\tr{\omega_t}{\omega}(x)\leq C$, and using \eqref{ma} we see that $\tr{\omega}{\omega_t}(x)\leq C$ and
$$Q_t\leq \log C -A\vp_t(x)+A\psi(x)\leq C,$$
thanks to \eqref{est2}. Therefore $Q_t\leq C$ holds on $X$ for all small $t>0$. Since $\sup_X\vp_t=0$, we have proved that
$$\tr{\omega}{\omega_t}\leq Ce^{-A\psi}.$$
This implies that $\omega_t$ is uniformly equivalent to $\omega$ on any compact set $K\subset X\backslash E_{nK}(\alpha)$, independent of $t$.
The higher order estimates (either the Calabi $C^3$ estimate \cite{Y} or the Evans-Krylov theory \cite{Si}) are local, and so we get uniform $C^\infty_{\mathrm{loc}}$ estimates for $\vp_t$ on $X\backslash E_{nK}(\alpha)$. It follows that
as $t\to 0$ the Ricci-flat metrics $\omega_t$ converge to $\omega_0=\alpha+\ddb\vp_0$ in $C^\infty_{\mathrm{loc}}(X\backslash E_{nK}(\alpha))$.
Note that this gives a slightly simpler proof of the fact that $\vp_0\in C^\infty(X\backslash E_{nK}(\alpha))$ than the one given in \cite[Theorem 4.1]{BEGZ}

Next, the metrics $\omega_t$ have a uniform upper bound on their diameter thanks to \cite{deg, ZT}. Once we show the Gromov-Hausdorff convergence statement, it will also follow that $(X\backslash E_{nK}(\alpha),\omega_0)$ has finite diameter, and hence is incomplete. The Gromov-Hausdorff convergence is proved exactly along the lines of \cite[Lemma 5.1]{RZ}, who dealt with the case when $X$ is projective and $[\alpha]$ is rational. The key point is that thanks to our main Theorem \ref{main}, $E_{nK}(\alpha)=\Null(\alpha)$, so the exact same argument as \cite[Lemma 5.1]{RZ} applies, with the subvariety $E$ there replaced by $\Null(\alpha)$.
\end{proof}
Finally, let us remark that the same proof as in Theorem \ref{contain} or Theorem \ref{finitet} shows that the metrics $\omega_t$
cannot converge smoothly to a K\"ahler metric near any point of $\Null(\alpha)$.

\end{document}